\documentclass[proceedings,submission%
]{dmtcs} 

\usepackage{amsmath}
\usepackage[usenames,dvipsnames]{color}
\usepackage{graphicx}
\usepackage{latexsym}
\usepackage[latin1]{inputenc}
\usepackage{mathtools}
\usepackage[shortlabels]{enumitem}
\usepackage{thmtools}
\usepackage{url}
\usepackage[all]{hypcap} 

\graphicspath{{figs/}}

\author[Mathieu Guay-Paquet and Alejandro H. Morales and Eric Rowland]{
  Mathieu Guay-Paquet\thanks{Supported by an NSERC Postdoctoral Fellowship} \and
  Alejandro H. Morales\thanks{Supported by a CRM-ISM Postdoctoral Fellowship} \and
  Eric Rowland}

\title{Structure and enumeration of $(3+1)$-free posets (extended abstract)}

\address{
  LaCIM,
  Universit\'e du Qu\'ebec \`a Montr\'eal,
  201 Pr\'esident-Kennedy,
  Montr\'eal QC~~H2X~3Y7, Canada}


\hypersetup{
  pdftitle={Structure and enumeration of (3+1)-free posets (extended abstract)},
  pdfauthor={Mathieu Guay-Paquet, Alejandro H. Morales, and Eric Rowland},
}

\newcommand{\QQ}{\rationals}

\newcommand{\JJ}{\mathcal{J}}

\newcommand{\Bl}{B_{\textnormal{lbl}}}
\newcommand{\Bu}{B_{\textnormal{unl}}}
\newcommand{\pl}{p_{\textnormal{lbl}}}
\newcommand{\pu}{p_{\textnormal{unl}}}
\newcommand{\Tl}{T_{\textnormal{lbl}}}
\newcommand{\Tu}{T_{\textnormal{unl}}}

\newcommand{\tbot}{\mathrel{\vcenter{\hbox{\includegraphics{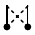}}}}}
\newcommand{\ttop}{\mathrel{\vcenter{\hbox{\includegraphics{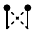}}}}}
\newcommand{\clone}{\approx}

\newcommand{\aut}{\mathop{\mathrm{Aut}}}

\newcommand{\abs}[1]{\left|{#1}\right|}

\newtheorem{theorem}{Theorem}[section]
\newtheorem{lemma}[theorem]{Lemma}
\newtheorem{proposition}[theorem]{Proposition}
\newtheorem{corollary}[theorem]{Corollary}

\newtheorem{definition}[theorem]{Definition}
\newtheorem{example}[theorem]{Example}

\newtheorem{remark}[theorem]{Remark}

\begin{document}
\maketitle
\begin{abstract}
\paragraph{Abstract.}
A poset is $(3+1)$-free if it does not contain the disjoint union of chains of length 3 and 1 as an induced subposet.
These posets are the subject of the $(3+1)$-free conjecture of Stanley and Stembridge.
Recently, Lewis and Zhang have enumerated \emph{graded} $(3+1)$-free posets, but until now the general enumeration problem has remained open.
We enumerate all $(3+1)$-free posets by giving a decomposition into bipartite graphs, and obtain generating functions for $(3+1)$-free posets with labelled or unlabelled vertices.

\paragraph{R\'esum\'e.}
Un poset sans $(3+1)$ est un poset qui n'a pas de sous-poset induit form\'e de deux cha\^{\i}nes disjointes de longeur 3 et 1.
Ces posets sont l'objet de la conjecture $(3+1)$ de Stanley et Stembridge.
R\'ecemment, Lewis et Zhang on \'enum\'er\'e les posets \emph{\'etag\'es} sans $(3+1)$, mais en g\'en\'eral la question d'\'enum\'eration est rest\'ee ouverte jusqu'\`a maintenant.
Nous \'enum\'erons tous les posets sans $(3+1)$ en donnant une d\'ecomposition de ces posets en graphes bipartis, et obtenons des fonctions g\'en\'eratrices qui les \'enum\`erent, qu'ils soient \'etiquet\'es ou non.
\end{abstract}

\section{Introduction}\label{sec:intro}

A poset $P$ is \emph{$(i+j)$-free} if it contains no induced subposet that is isomorphic to the poset consisting of two disjoint chains of lengths $i$ and $j$.
In particular, $P$ is $(3+1)$-free if there are no vertices $a, b, c, d \in P$ such that $a < b < c$ and $d$ is incomparable to $a$, $b$, and $c$.

Posets that are $(3+1)$-free play a role in the study of
Stanley's chromatic symmetric function~\cite{St1, St2}, a symmetric
function associated with a poset that generalizes the chromatic
polynomial of a graph.
Namely, a well-known conjecture of Stanley and Stembridge~\cite{StSt} is that the chromatic symmetric function of a $(3+1)$-free poset has positive coefficients in the basis of elementary symmetric functions. As evidence toward this conjecture, Stanley~\cite{St1} verified the conjecture for the class of $3$-free posets, and Gasharov~\cite{G2} has shown the weaker result that the chromatic symmetric function of a $(3+1)$-free poset is Schur-positive.

To make more progress toward the Stanley--Stembridge conjecture, a
better understanding of $(3+1)$-free posets is needed. Reed and
Skandera~\cite{Sk1, SkR} have given structural results and a
characterization of $(3+1)$-free posets in terms of their antiadjacency matrix. 
In addition, certain families of $(3+1)$-free posets have been
enumerated. For example, the number of $(3+1)$-and-$(2+2)$-free
posets with $n$ vertices is the $n$th Catalan number~\cite[Ex.~6.19(ddd)]{EC2};
Atkinson, Sagan and Vatter~\cite{ASV} have enumerated the permutations
that avoid the patterns $2341$ and $4123$, which give rise to the 
$(3+1)$-free posets of dimension two; and Lewis and
Zhang~\cite{LZ} have made significant progress by enumerating \emph{graded} $(3+1)$-free posets in terms of bicoloured graphs\footnote{Throughout this paper, a \emph{bicoloured} graph is a bipartite graphs with a specified ordered bipartition. For example, there are 2 bicoloured graphs with 1 vertex, 6 bicoloured graphs with 2 labelled vertices, and 4 bicoloured graphs with 2 unlabelled vertices.} using a new structural decomposition.
However, until now the general enumeration problem for $(3+1)$-free posets remained open~\cite[Ex.~3.16(b)]{EC1}.

In this paper, we give generating functions for
$(3+1)$-free posets with unlabelled and labelled vertices in terms
of the generating functions for bicoloured graphs with unlabelled and
labelled vertices, respectively. As in the graded case, the two problems are equally hard, although the enumeration problem for bicoloured graphs has received more attention.

In the unlabelled case, let $\pu(n)$ be the number of $(3+1)$-free posets with $n$ unlabelled vertices, and let $S(c,t)$ be the unique formal power series solution (in $c$ and $t$) of the cubic equation
\begin{equation}\label{recurrTandC}
  S(c,t) = 1 + \frac{c}{1+c}S(c,t)^2 + tS(c,t)^3.
\end{equation}
We show that the ordinary generating function for unlabelled $(3+1)$-free posets is
\begin{equation}\label{ordgs}
  \sum_{n\geq 0} \pu(n) x^n
    = S\big(x/(1-x), 1-2x-\Bu(x)^{-1}\big),
\end{equation}
where $\Bu(x) = 1 + 2x + 4x^2 + 8x^3 + 17x^4 + \cdots$ is the ordinary generating function for unlabelled bicoloured graphs. Before our investigation, the On-Line Encyclopedia of Integer Sequences~\cite{OEIS} had 22 terms in the entry~\cite[\href{http://oeis.org/A049312}{A049312}]{OEIS} for the coefficients of $\Bu(x)$, but only 7 terms in the entry~\cite[\href{http://oeis.org/A079146}{A079146}]{OEIS} for the numbers $\pu(n)$. Using~\eqref{ordgs}, we have closed this gap; the numbers $\pu(n)$ for $n = 0, 1, 2, \ldots, 22$ are
\begin{quote}
1, 1, 2, 5, 15, 49, 173, 639, 2469, 9997, 43109, 205092, 1153646,
8523086, 91156133, 1446766659, 32998508358, 1047766596136,
45632564217917, 2711308588849394,
\linebreak
 219364550983697100, 24151476334929009951, 3618445112608409433287.
\end{quote}

Similarly, in the labelled case, let $\pl(n)$ be the number of $(3+1)$-free posets with $n$ labelled vertices. We show that the exponential generating function for labelled $(3+1)$-free posets is
\begin{equation}\label{expgs}
  \sum_{n\geq 0} \pl(n) \frac{x^n}{n!}
    = S\big(e^x-1, 2e^{-x} -1 - \Bl(x)^{-1}\big),
\end{equation}
where $\Bl(x) = \sum_{n\geq 0} \sum_{i=0}^n \binom{n}{i} 2^{i(n-i)} \frac{x^{n}}{n!}$ is the exponential generating function for labelled bicoloured graphs. Such bicoloured graphs are easy to count, but before our investigation the OEIS had only 9 terms in the entry~\cite[\href{http://oeis.org/A079145}{A079145}]{OEIS} for $\pl(n)$. Using~\eqref{expgs}, arbitrarily many terms $\pl(n)$ can be computed.

Our main tool is a new decomposition of $(3+1)$-free posets into \emph{parts} (called \emph{clone sets} and \emph{tangles}).
This \emph{tangle decomposition} is compatible with the automorphism group, in the sense that for a $(3+1)$-free poset $P$, $\aut(P)$ breaks up as the direct product of the automorphism groups of its parts.
The tangle decomposition also generalizes a decomposition of Reed and Skandera~\cite{SkR} for $(3+1)$-and-$(2+2)$-free posets given by \emph{altitudes} of vertices.
In terms of generating functions, the restriction of our results to $(3+1)$-and-$(2+2)$-free posets corresponds to the specialization $t=0$ in~\eqref{recurrTandC}. Indeed, one can see that $S(x/(1-x),0)$ satisfies the functional equation for the Catalan generating function, which is consistent with the enumeration result stated earlier for $(3+1)$-and-$(2+2)$-free posets~\cite[Ex.~6.19(ddd)]{EC2}.

\begin{remark}
Using the tangle decomposition it is possible to quickly generate all $(3+1)$-free posets of a given size up to isomorphism in a straightforward way (see \autoref{cor:data}). With this approach, we were able to list all $(3+1)$-free posets on up to 11 vertices in a few minutes on modest hardware.
Note that this technique can accommodate the generation of interesting subclasses of $(3+1)$-free posets (\textit{e.g.}, $(2+2)$-free, weakly graded, strongly graded, co-connected, fixed height) or constructing these posets from the bottom up, level by level (which can help compute invariants like the chromatic symmetric function).
\end{remark}

\begin{remark}
Comparing the list of numbers above with data provided by Joel Brewster
Lewis for the number of graded $(3+1)$-free posets~\cite[\href{http://oeis.org/A222863}{A222863}, \href{http://oeis.org/A222865}{A222865}]{OEIS}, it appears that, asymptotically,
almost all $(3+1)$-free posets are graded.
We prove this in the full version of this paper \cite{GPMR}, building on the asymptotic analysis of Lewis and Zhang for the graded $(3+1)$-free posets.
In fact, almost all $(3+1)$-free posets are $3$-free, so their Hasse diagrams are bicoloured graphs.
\end{remark}

\noindent\textsc{Outline.}
In \autoref{sec:partition}, we describe the
tangle decomposition of a $(3+1)$-free poset into clone sets and
tangles and use it to compute the poset's automorphism group. In
\autoref{sec:skeleta}, we describe the relationships between the
different clone sets and tangles of a $(3+1)$-free poset as parts of a
structure called the \emph{skeleton} and enumerate the possible
skeleta. In \autoref{sec:genfunc}, we enumerate tangles in terms of
bicoloured graphs, and as a result we obtain generating functions for
$(3+1)$-free posets.

\section{The tangle decomposition}\label{sec:partition}

Throughout the paper, we assume that $P$ is a $(3+1)$-free poset.
We write $a \parallel b$ if vertices $a$ and $b$ in a poset are incomparable.
In this section, we describe the tangle decomposition of a $(3+1)$-free poset.

Given a vertex $a \in P$, we write $D_a = \{x \in P : x < a\}$ and $U_a = \{x \in P : x > a\}$ for the (strict) downset and upset of $a$.
The set $\JJ(P)$ of all downsets of $P$ (that is, all downward closed subsets of $P$, not just those of the form $D_a$ for some $a\in P$) forms a distributive lattice, and in particular a poset, under set inclusion.
Similarly, the set of upsets of $P$ forms a poset under set inclusion, but it will be convenient for us to consider instead the complements $P \setminus U_a \in \JJ(P)$ of the upsets of vertices $a \in P$.

\begin{definition}
The \emph{view} $v(a)$ from a vertex $a \in P$ is the pair $(D_a, P \setminus U_a) \in \JJ(P) \times \JJ(P)$.
If $v(a) = v(b)$, then we say $a$ and $b$ are \emph{clones} and write $a \clone b$.
\end{definition}

Note that the set $v(P)$ of views of all vertices of $P$ inherits a poset structure from the set $\JJ(P) \times \JJ(P)$, where $v(a) \leq v(b)$ if and only if $D_a \subseteq D_b$ and $U_a \supseteq U_b$.

Also note that two vertices $a, b \in P$ are clones precisely when they are \emph{interchangeable}, in the sense that the permutation of the vertices of $P$ which only exchanges $a$ and $b$ is an automorphism of $P$.

\begin{example}
  \autoref{example poset} shows a $(3+1)$-free poset and its view poset. Since $v(d) = v(e)$, we have $d \clone e$.
\end{example}

\begin{remark}
The notion of clones is related to the notion of {\em trimming} of Lewis and Zhang~\cite{LZ}. Also, Zhang~\cite{YXZ} has used techniques involving clones and $(2+2)$-avoidance to prove enumeration results about families of graded posets.
\end{remark}

\begin{definition}
Let $a, b \in P$.
We write $a \ttop b$ if $D_a \parallel D_b$, and we write $a \tbot b$ if $U_a \parallel U_b$.
\end{definition}

\begin{figure}[t]
\[
\vcenter{\hbox{\includegraphics{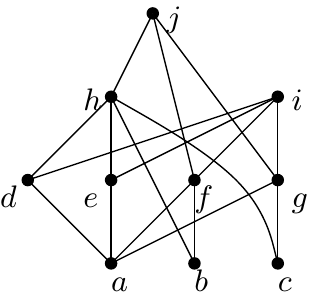}}}
\quad
\begin{array}{r@{{}=(\{}l@{,\{}l@{)}}
v(j)&abcdefgh\}   &abcdefghij\} \\   
v(i)&abcdefg\}   &abcdefghi\}    \\
v(h)&abcde\}   &abcdefghi\}    \\
v(g)&ac\}   &abcdefgh\}    \\
v(f)&ab\}   &abcdefgh\}    \\
v(e)&a\} &abcdefg\}    \\
v(d)&a\}   &abcdefg\}    \\
v(c)&\}   &abcdef\}   \\
v(b)&\}   &abcde\}   \\
v(a)&\}  & abc\}    
\end{array}
\quad
\vcenter{\hbox{\includegraphics{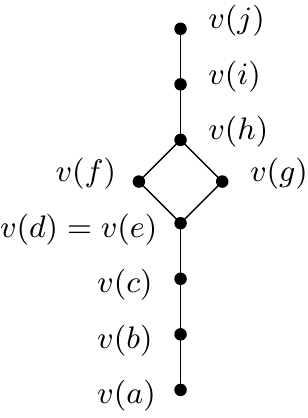}}}
\]
\caption{Left: the Hasse diagram of a $(3+1)$-free poset $P$ with $10$
  vertices. Centre: the list of views of the vertices of $P$. Right: the view poset $v(P)$.}
\label{example poset}
\end{figure}

\pagebreak

The idea behind the notation is the following.
If $a \ttop b$, then there is some vertex $c \in D_a \setminus D_b$, so that $c < a$ and $c \not< b$, and there is some $d \in D_b \setminus D_a$, so that $d \not< a$ and $d < b$. Then, it can be checked that $a, b, c, d$ are distinct vertices, and that they are incomparable except for the two relations $c < a$ and $d < b$. Hence we have the following induced $(2+2)$ subposet with $a$ and $b$ on the top:
\begin{center}
	\includegraphics{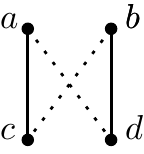}
\end{center}
Dually, if $a \tbot b$ then there is an induced $(2+2)$ subposet with $a$ and $b$ on the bottom.

\begin{example}
  In \autoref{example poset}, we have $f \ttop g$ and $b \tbot c$, but \emph{not} $a \tbot b$, since $U_b \subseteq U_a$.
\end{example}

The following lemma records basic properties of the relations $\clone$, $\ttop$, and $\tbot$ and their interactions.

\begin{lemma}\label{lem:relations}
  Let $P$ be a $(3+1)$-free-poset, and let $a, b, c$ be any vertices of $P$.
  \begin{enumerate}[(i),nosep]
    \item If $a \clone b$ and $b \clone c$, then $a \clone c$.
    \item If $a \ttop b$ and $b \clone c$, then $a \ttop c$.
    \item If $a \tbot b$ and $b \clone c$, then $a \tbot c$.
    \item If $a \ttop b$, then $U_a = U_b$.
    \item If $a \tbot b$, then $D_a = D_b$.
    \item We have $v(a) \parallel v(b)$ if and only if $a \ttop b$ or $a \tbot b$.
    \item\label{item:rel4} It is not the case that both $a \ttop b$ and $b \tbot c$.
  \end{enumerate}
\end{lemma}

Now, consider a graph $\Gamma$ on the vertices of $P$ with edge set $\{(a, b) : a \ttop b\}$.
We say that a subset $A \subseteq P$ is the \emph{top of a tangle} if $\abs{A} \geq 2$ and $A$, when viewed as a subset of $V(\Gamma)$, is a connected component of $\Gamma$.
Analogously, a subset $B \subseteq P$ is the \emph{bottom of a tangle} if $\abs{B} \geq 2$ and $B$ is a connected component under the relation $\tbot$.

By conclusion~\ref*{item:rel4} of \autoref{lem:relations}, if $A$ is the top of a tangle and $B$ is the bottom of a tangle, then $A \cap B = \emptyset$. Let us say that a top of a tangle $A$ and a bottom of a tangle $B$ are \emph{matched} if there is an induced $(2+2)$ subposet whose top two vertices are in $A$, and whose bottom two vertices are in $B$.

\begin{proposition}\label{prop:matching}
  In a $(3+1)$-free poset $P$, every top of a tangle is matched to a unique bottom of a tangle, and every bottom of a tangle is matched to a unique top of a tangle. That is, there is a perfect matching between tops of tangles and bottoms of tangles of $P$.
\end{proposition}

\autoref{prop:matching} justifies the terms `top of a tangle' and `bottom of a tangle' and the following definition.

\begin{definition}
  A \emph{tangle} is a matched pair $T = (A, B)$ of a top of a tangle $A$ and a bottom of a tangle $B$.
\end{definition}

In other words, a tangle is a subposet of $P$ that is connected by induced $(2+2)$ subposets.
In particular, $P$ is $(2+2)$-free exactly when it has no tangles.

\begin{example}
  Very often, a two-level poset which is not connected consists of a single tangle. For example, let $P$ be the poset with vertices $\{a_1, a_2, a_3, c_1, c_2\} \cup \{b, d\}$ and relations $a_i > c_j$, $b > d$. Then, the connected components of $P$ are $\{a_1, a_2, a_3, c_1, c_2\}$ and $\{b, d\}$. Every subset of the form $\{a_i, b, c_j, d\}$ forms an induced $(2+2)$ subposet, so $\{a_1, a_2, a_3, b\}$ is the top of a tangle, $\{c_1, c_2, d\}$ is the bottom of a tangle, and the whole poset $P$ is a single tangle.
\end{example}

\begin{example}
  In the poset $P$ of \autoref{example poset}, the connected component of $f$ under $\ttop$ is $\{f, g\}$, and the connected component of $b$ under $\tbot$ is $\{b, c\}$.  Therefore $P$ contains the tangle $T = (\{f, g\}, \{b, c\})$.  One can check that in fact this is the only tangle of $P$.
\end{example}

\begin{definition}
  Let $T_1 = (A_1, B_1), \ldots, T_s = (A_s, B_s)$ be the tangles of $P$.
  A \emph{clone set} is an equivalence class, under $\clone$, of vertices in $P \setminus \bigcup_{j=1}^s (A_j \cup B_j)$.
  We refer to tangles and clone sets as \emph{parts} of $P$.
The set of parts is the \emph{tangle decomposition} of $P$.
\end{definition}

\begin{example}
  The tangle decomposition of the poset in \autoref{example poset} appears in \autoref{small example}.  It consists of six parts\hspace{0.08em}---\hspace{0.08em}five clone sets and one tangle.
\end{example}

\begin{figure}[!h]
\begin{center}
\includegraphics{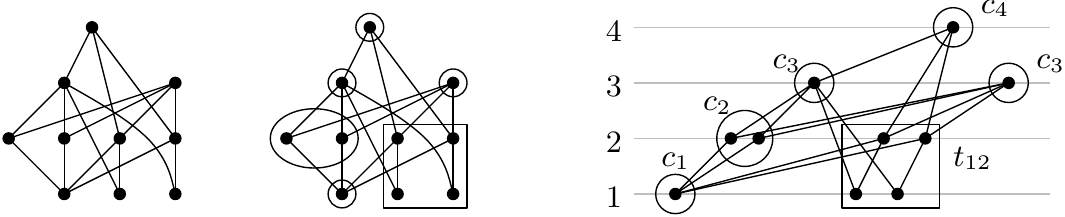}
\caption{Left: the Hasse diagram of the poset $P$ from \autoref{example poset}. Centre: the tangle decomposition of $P$ into its parts. Right: a compatible listing of the parts.
Clone sets are enclosed in circles, and tangles are enclosed in boxes.}
\label{small example}
\end{center}
\end{figure}

The tangle decomposition provides a decomposition of a $(3+1)$-free poset from which the automorphism group, among other properties, can be computed. To show this, it will be useful to have a different characterization of the tops of tangles, bottoms of tangles, and clone sets of $P$ which gives a natural ordering of these subsets of $P$, as follows.
A \emph{co-connected component} of a poset $Q$ is a connected component of the incomparability graph of $Q$.

\begin{proposition}\label{preimages}
  Let $v(P) \subseteq \JJ(P) \times \JJ(P)$ be the poset of views of all vertices of the $(3+1)$-free poset $P$.
Then, there is a listing $(S_1, S_2, \ldots, S_k)$ of the co-connected components of $v(P)$ such that for every $x \in S_i$ and every $y \in S_{i+1}$, we have $x < y$. Moreover, the preimages $v^{-1}(S_i)$ for $i = 1, 2, \ldots, k$ are exactly the tops of tangles, bottoms of tangles, and clone sets of $P$.
\end{proposition}

Let $\aut(P)$ be the automorphism group of the poset $P$.
Any part $X_i$ of $P$ gives an induced subposet of $P$, and we write
$\aut(X_i)$ for its automorphism group as a poset. In particular, if $X_i$
is a clone set with $k$ vertices, then $\aut(X_i)$ is the symmetric
group on these $k$ vertices; if $X_i$ is a tangle, then it can be seen as a bicoloured graph (with colour classes `top' and `bottom'), and $\aut(X_i)$ is the group of colour-preserving automorphisms of this graph.

\begin{theorem}
  Let $P$ be a $(3+1)$-free poset, decomposed into its clone sets $C_1, C_2, \ldots, C_r$ and its tangles $T_1, T_2, \ldots, T_s$. Then, the automorphism group of $P$ is
  \[
    \aut(P) = \prod_{i = 1}^r \aut(C_i) \times \prod_{j = 1}^s \aut(T_j).
  \]
\end{theorem}

Note that the tangle decomposition of a $(3+1)$-free poset $P$ into its parts generalizes the decomposition considered by Reed and Skandera~\cite{SkR} of a $(3+1)$-and-$(2+2)$-free poset given by the \emph{altitude} $\alpha(a) = \abs{D_a} - \abs{U_a}$ of the vertices $a \in P$, since the altitude $\alpha(a)$ is a function of the view $v(a)$. Of course, even in a $(3+1)$-free poset $P$ with an induced $(2+2)$ subposet, the altitude is well-defined, and it gives a finer decomposition of $P$ than the tangle decomposition.
However, the altitude decomposition is too fine, as the example in \autoref{fig:tau} shows.
Namely, there is an automorphism $\tau$ which swaps the two vertices with altitude $-1$, the two vertices with altitude $-2$, and two of the three vertices with altitude $2$, as illustrated. But there is no automorphism which acts nontrivially on a single block of the altitude decomposition.

In contrast, for the tangle decomposition, every automorphism of the poset can be factored as a product of automorphisms which only act nontrivially on a single part.

\begin{figure}[h]
  \begin{center}
    \includegraphics[scale=.8]{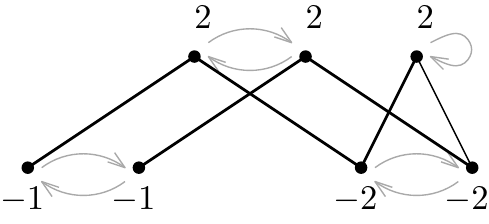}
    \caption{A poset $P$ consisting of a single tangle. The vertices are labelled by their altitude $\alpha$, and the arrows describe an automorphism $\tau$ of $P$.}
    \label{fig:tau}
  \end{center}
\end{figure}

\section{Skeleta}\label{sec:skeleta}

Any finite poset $P$ can be decomposed into \emph{levels} as follows: take $L_1$ to be the set of minimal vertices of $P$, $L_2$ to be the set of subminimal vertices (that is, the set of minimal vertices of $P \setminus L_1$), and so on up to the set $L_h$ of sub$^{(h-1)}$minimal vertices of $P$, where $h$ is the \emph{height} of $P$. We say that the \emph{level} of a vertex $a \in P$ is $\ell(a)$, where $a \in L_{\ell(a)}$.

If $P$ is $(3+1)$-free, then the only interesting part of the poset structure occurs between adjacent levels, as the following proposition shows.

\begin{proposition}[Lewis and Zhang~\cite{LZ}]\label{prop:two-levels}
  Let $P$ be a $(3+1)$-free poset and $a, b \in P$ be two vertices with $\ell(a) \leq \ell(b) - 2$. Then, we have $a < b$.
\end{proposition}

Note that the covering relations of $P$ may include relations $a < b$ for which $\ell(a) = \ell(b) - 2$.
This occurs in \autoref{example poset}, for example, where $b < h$, $c < h$, $f < j$, and $g < j$ are covering relations.

\pagebreak

The following proposition gives a partial converse of \autoref{prop:two-levels}.

\begin{proposition}[Reed and Skandera~\cite{SkR}]\label{prop:partial-converse}
  Let $P$ be a poset such that for any two vertices $a, b \in P$ with
  $\ell(a) \leq \ell(b) - 2$, we have $a < b$. Then, $P$ is
  $(3+1)$-free if and only if for any two vertices $c, d \in P$ with
  $\ell(c) = \ell(d)$, we have $U_c \subseteq U_d$ or $D_c \subseteq
  D_d$ (and symmetrically, $U_c \supseteq U_d$ or $D_c \supseteq D_d$).
\end{proposition}

Note that the vertices of a clone set $C_i$ all have the same downset, so they are on the same level. Also, any copy of the $(2+2)$ poset must be contained in two adjacent levels, so any tangle $T_j$ must be contained in two adjacent levels. Thus, we can speak of the level of a clone set or the (adjacent) levels of a tangle.

By construction, the poset structure between two parts of $P$ is fairly restricted. If $C_i$ and $C_j$ are distinct clone sets, then $C_i$ is either completely above, completely below, or completely incomparable with $C_j$ (meaning that every vertex of $C_i$ has the same relationship with every vertex of $C_j$). If $C_i$ is a clone set and $T_j$ is a tangle, then $C_i$ can be
\begin{itemize}[noitemsep]
  \item completely above $T_j$;
  \item completely above the bottom of $T_j$ and incomparable with the top;
  \item completely below the top of $T_j$ and incomparable with the bottom;
  \item completely below $T_j$; or
  \item completely incomparable with $T_j$.
\end{itemize}
Similarly, there are only six possible ways for two tangles $T_i$ and $T_j$ to relate to each other. The following theorem shows how all of these relationships between different parts of $P$ can be put together.

\begin{theorem}\label{thm:listing}
  Let $P$ be a $(3+1)$-free poset, decomposed into clone sets $C_1, \ldots,\allowbreak C_r$ and tangles $T_1, \ldots,\allowbreak T_s$. Then, there exists a listing $(X_1, \ldots,\allowbreak X_{r+s})$ of the clone sets and the tangles of $P$ such that, for any two vertices $a \in X_i$ and $b \in X_j$ with $i \neq j$, we have $a < b$ exactly when
  \begin{enumerate}[(i),nosep]
    \item $\ell(a) \leq \ell(b) - 2$; or
    \item\label{itm:listing-order} $\ell(a) = \ell(b) - 1$ and $i < j$.
  \end{enumerate}
\end{theorem}

\begin{definition}
  A listing which satisfies the conditions of \autoref{thm:listing} is called a \emph{compatible} listing.
\end{definition}

\begin{example}
  A compatible listing for the poset in \autoref{example poset} is ${\big (}\{a\},\allowbreak \{d, e\},\allowbreak \{h\},\allowbreak (\{f, g\},\allowbreak \{b, c\}),\allowbreak \{j\},\allowbreak \{i\}{\big)}$, as shown in \autoref{small example}.
\end{example}

\begin{proof}[idea for \autoref{thm:listing}]
For each level, we can get a partial listing of the parts which intersect $L_i$ according to their positions on the view poset $v(P)$.
Then, the listing for $L_i$ and $L_{i+1}$ can be interleaved in a unique way to respect condition~\ref*{itm:listing-order}, so it follows that all of them can be reconciled into a single compatible listing.
\end{proof}

Note that the listing $(X_1, X_2, \ldots, X_{r+s})$ from
\autoref{thm:listing} is not unique in general. In particular, if
$(\ldots, X_i, X_{i+1}, \ldots)$ is a compatible listing, then the
listing $(\ldots, X_{i+1}, X_i, \ldots)$ obtained by swapping the parts $X_i$
and $X_{i+1}$ is compatible exactly when $X_i$ and $X_{i+1}$ contain
no vertices on the same or on adjacent levels of $P$.
We call such a swap \emph{valid}.
 
\begin{example}
In \autoref{small example} we can swap the clone set $\{j\}$ on level 4 with the
tangle $(\{f, g\}, \{b, c\})$ on levels 1 and 2 to obtain another compatible listing for the poset.
\end{example}

Therefore the natural setting for compatible listings is that of free partially commuting monoids~\cite{CF}, also known as trace monoids~\cite{VCT}.

\begin{definition}
  Let $\Sigma$ be the countable alphabet
  \[
    \Sigma = \{c_1, c_2, \ldots, c_i, \ldots\} \cup \{t_{12}, t_{23}, \ldots, t_{i\,i+1}, \ldots\},
  \]
  let $\Sigma^*$ be the free monoid generated by $\Sigma$, and let $M$ be the free partially commuting monoid with commutation relations
  \begin{alignat*}{2}
    c_i c_j &= c_j c_i, &\qquad&\text{if $\abs{i - j} \geq 2$}, \\
    c_i t_{j\,j+1} &= t_{j\,j+1} c_i, &&\text{if $i \leq j - 2$ or $i \geq j + 3$}, \\
    t_{i\,i+1} t_{j\,j+1} &= t_{j\,j+1} t_{i\,i+1}, &&\text{if $\abs{i - j} \geq 3$}.
  \end{alignat*}
\end{definition}

\begin{definition}
  If $P$ is a $(3+1)$-free poset, then for each compatible listing $(X_1, X_2, \ldots, X_{r+s})$ of its clone sets and tangles, we can obtain a word in $\Sigma^*$ by replacing each clone set at level $i$ by the letter $c_i$ and each tangle straddling levels $\{i, i+1\}$ by the letter $t_{i\,i+1}$.
It can be seen that any two compatible listings for $P$ are related by a sequence of valid swaps, so the set of these words is an equivalence class under the commutation relations for $M$ (see, \textit{e.g.}, \cite[Chapter~1]{VCT}), and the corresponding element of $M$ is called the \emph{skeleton} of $P$.
\end{definition}

\begin{example}\label{ex:skeleta}
  The two representatives in $\Sigma^*$ for the skeleton of the poset in \autoref{small example} are $c_1 c_2 c_3 t_{12} c_4 c_3$ and $c_1 c_2 c_3 c_4 t_{12} c_3$.
\end{example}

The point of a skeleton is that it exactly captures the relationships between different parts of $P$. More precisely, two posets with the same skeleton and isomorphic parts are themselves isomorphic; conversely, given a skeleton, any set of parts (with the right number of clone sets and tangles) can be plugged into the skeleton. Together, \autoref{cor:data}, \autoref{thm:skel-char}, and \autoref{thm:skel-char-lex} below show this and give a characterization of the elements of $M$ which are skeleta.

\begin{corollary}\label{cor:data}
  Let $P$ be a $(3+1)$-free poset. Then, $P$ is uniquely determined (up to isomorphism) by its skeleton together with, for each letter $c_i$ or $t_{i\,i+1}$ of the skeleton, the cardinality of the corresponding clone set or the isomorphism class of the corresponding tangle.
\end{corollary}

\begin{theorem}\label{thm:skel-char}
  Let $m$ be an element of the monoid $M$. Then, $m$ is the skeleton of some $(3+1)$-free poset if and only if
  \begin{enumerate}[(i),nosep]
    \item\label{itm:staircase}
      every representative $w \in \Sigma^*$ for $m$ starts with the letter $c_1$ or $t_{12}$; and
    
    \item\label{itm:no-repeat}
      no representative $w \in \Sigma^*$ for $m$ contains a factor of the form $c_i c_i$, $i \geq 1$.
  \end{enumerate}
\end{theorem}

Note that condition~\ref*{itm:staircase} of \autoref{thm:skel-char} corresponds to the requirement that every vertex of $P$ on level $L_{i + 1}$ be greater than some vertex on the previous level $L_i$, while condition~\ref*{itm:no-repeat} forbids pairs of clone sets that could be merged into a single clone set.

\begin{figure}[!b]
\centering
\fbox{\begin{minipage}{.95\linewidth}
Consider the $26$-vertex $(3+1)$-free poset $P$ with $10$ parts shown in the compatible listing below. Only some of the comparability and incomparability relations between parts are drawn, but the others can be determined from \autoref{thm:listing}.
\[\includegraphics[scale=.75]{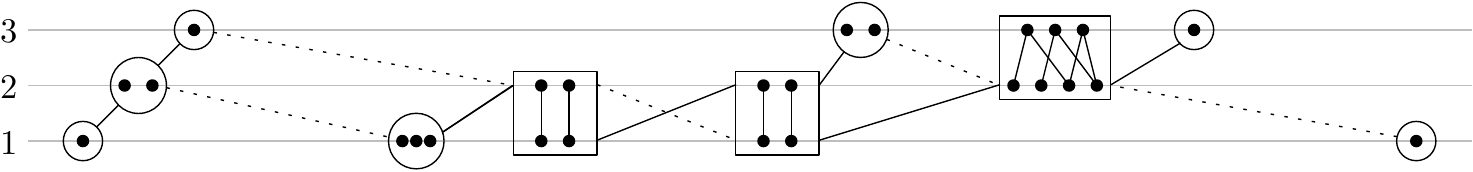}\]
The word $w_0 = c_1 c_2 c_3 c_1 t_{12} t_{12} c_3 t_{23} c_3 c_1$, shown below in a suggestive manner, is the lexicographically maximal representative for the skeleton of $P$.
\[\includegraphics[scale=.75]{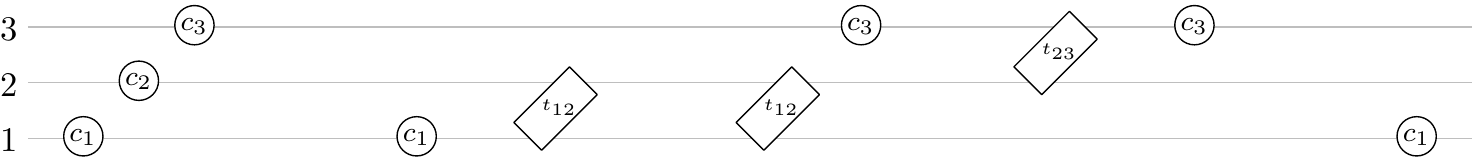}\]
The decorated Dyck path associated with $w_0$ is the following.
\[\includegraphics[scale=.75]{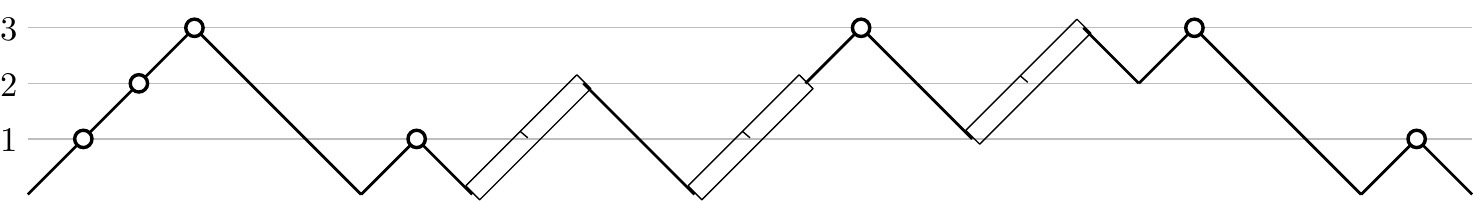}\]
\end{minipage}}
\caption{An example of the bijection given in \autoref{thm:skeleta-bijection}.}
\label{fig:dyck}
\end{figure}

\begin{theorem}\label{thm:skel-char-lex}
  Let $m$ be an element of the monoid $M$. Then, there exists a representative $w_0 \in \Sigma^*$ for $m$ for which every pair of consecutive letters is either
  \begin{alignat*}{2}
    & c_i c_j &\qquad&\text{for $i \geq j - 1$; or} \\
    & c_i t_{j\,j+1} &&\text{for $i \geq j - 1$; or} \\
    & t_{i\,i+1} c_j &&\text{for $i \geq j - 2$; or} \\
    & t_{i\,i+1} t_{j\,j+1} &&\text{for $i \geq j - 2$.}
  \end{alignat*}
  Furthermore,
  \begin{enumerate}[(i),nosep]
    \item
      this representative $w_0$ is unique and is the lexicographically maximal representative for $m$ with respect to the total order $\{c_1 < t_{12} < c_2 < t_{23} < \cdots\}$ on $\Sigma$;
    
    \item
      if $w_0$ starts with $c_1$ or $t_{12}$, then every representative $w \in \Sigma^*$ for $m$ starts with $c_1$ or $t_{12}$; and
    
    \item
      if $w_0$ does not contain a factor of the form $c_i c_i$, $i \geq 1$, then no representative $w \in \Sigma^*$ for $m$ contains a factor of this form.
  \end{enumerate}
\end{theorem}

\begin{example}
  Of the two representatives given in \autoref{ex:skeleta}, $c_1 c_2 c_3 c_4 t_{12} c_3$ is lexicographically maximal.
\end{example}

Using this characterization of skeleta, we can enumerate them, and this will allow us to obtain generating functions for $(3+1)$-free posets.

\begin{theorem}\label{thm:skeleta-bijection}
  There is a bijection between skeleta of $(3+1)$-free posets and certain decorated Dyck paths. (See \autoref{fig:dyck} for an example.)
\end{theorem}

\begin{proof}
  Given the lexicographically maximal representative $w_0$ for a skeleton, we can obtain a decorated Dyck path that starts at $(0,0)$, ends at $(2n,0)$ for some $n \geq 0$, and never goes below the $x$-axis as follows: replace each letter $c_i$ by a $(1, 1)$ step ending at height $i$, each letter $t_{i\,i+1}$ by a $(2, 2)$ step ending at height $i+1$, and add $(1,-1)$ down steps as necessary.
  We call the result \emph{decorated} since a $(2, 2)$ step can be seen as a pair of consecutive decorated $(1, 1)$ steps.
Since $w_0$ not contain $c_i c_i$ as a factor, the decorated Dyck path obtained from $w_0$ contains no sequence $(1, 1), (1, -1), (1, 1)$ of consecutive undecorated steps (up-down-up).
Conversely, every decorated Dyck path avoiding this sequence can be obtained from a skeleton.
\end{proof}

\begin{theorem}\label{thm:skeleta-genfunc}
  Let $S(c, t) \in \QQ[[c, t]]$ be the ordinary generating function for skeleta with respect to the number of clone sets and the number of tangles, that is, the formal power series
  \[
    S(c, t) = \sum_{r, s \geq 0} \left(\text{\# of distinct skeleta with $r$ clone sets and $s$ tangles}\right) c^r t^s.
  \]
  Then, $S(c, t)$ is uniquely determined by the equation
  \begin{equation}\label{eq:RecSkel}
    S(c, t) = 1 + \frac{c}{1 + c} S(c, t)^2 + t S(c, t)^3.
  \end{equation}
\end{theorem}

\begin{proof}[idea]
See \autoref{figDecomp}.
\end{proof}

\begin{figure}[!h]
\begin{center}
\begin{align*}
  \left\{\vcenter{\hbox{\includegraphics[scale=0.65]{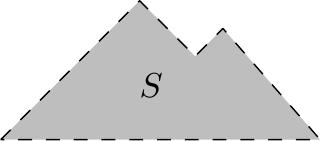}}}\right\}
    &= \{\varepsilon\}
      \cup \left\{\vcenter{\hbox{\includegraphics[scale=0.65]{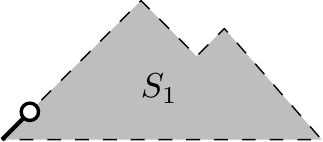}}}\right\}
      \cup \left\{\vcenter{\hbox{\includegraphics[scale=0.65]{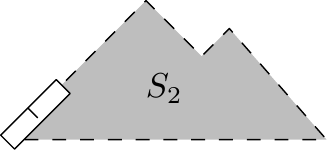}}}\right\} \\
  \left\{\vcenter{\hbox{\includegraphics[scale=0.65]{decomp2}}}\right\}
    &= \left\{\vcenter{\hbox{\includegraphics[scale=0.65]{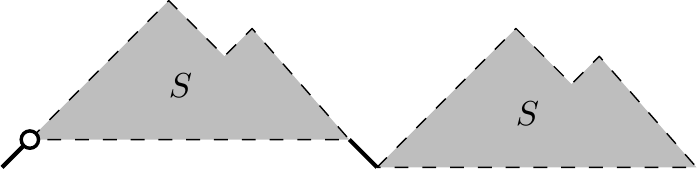}}}\right\}
      \setminus \left\{\vcenter{\hbox{\includegraphics[scale=0.65]{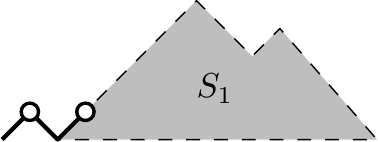}}}\right\} \\
  \left\{\vcenter{\hbox{\includegraphics[scale=0.65]{decomp3}}}\right\}
    &= \left\{\vcenter{\hbox{\includegraphics[scale=0.65]{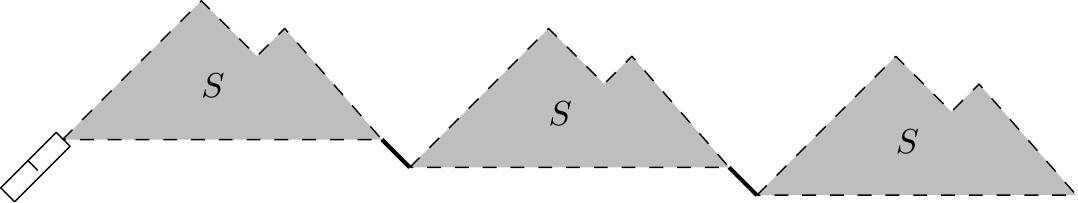}}}\right\}
\end{align*}
\caption{Equations relating the sets counted by $S(c, t)$, $S_1(c, t)$, and $S_2(c, t)$, where $S_1(c, t)$ and $S_2(c, t)$ are the generating functions for decorated Dyck paths beginning with $(1, 1)$ and $(2, 2)$, respectively.}
\label{figDecomp}
\end{center}
\end{figure}

\section{Enumeration}\label{sec:genfunc}

In this section, we carry out the enumeration of unlabelled and
labelled $(3+1)$-free posets by reducing it to the enumeration of
unlabelled and labelled bicoloured graphs. Our approach is to consider such a bicoloured graph as a $(3+1)$-free poset in the natural way (with colour classes `top' and `bottom') and to apply the machinery of \autoref{sec:skeleta}, as shown in the following lemma.

\begin{lemma} \label{lem:masterSkelBip}
The ordinary generating function for skeleta of bicoloured graphs is given by
 \begin{multline}
\label{eq:BipTangles}
\sum_{r_1,r_2,s \geq 0} 
 \left(\text{\parbox{.42\linewidth}{\# of
        skeleta of bicoloured graphs with $r_1$ clone sets on
        level $1$, $r_2$ clone sets on level $2$, and $s$
        tangles}}\right) c_1^{r_1} c_2^{r_2} t_{12}^s \\
  = \left(1-\frac{c_1}{1+c_1} - \frac{c_2}{1+c_2} - t_{12}\right)^{-1}.
\end{multline}
\end{lemma}

Now that we have an explicit expression for the generating function of
skeleta of bicoloured graphs, we can perform
appropriate substitutions to get equations relating the generating functions
for tangles and for bicoloured graphs.

\begin{theorem}
\label{thm:genfcnbipunlab}
  Let $\Bu(x, y) \in \QQ[[x, y]]$ be the ordinary generating function for unlabelled bicoloured graphs, up to isomorphism. Then, the ordinary generating function for unlabelled tangles is
  \[
    \Tu(x, y) = 1 - x - y - \Bu(x, y)^{-1}.
  \]
\end{theorem}

\begin{proof}
This follows from \autoref{lem:masterSkelBip} by plugging in the values $c_1= x/(1-x)$ and $c_2=y/(1-y)$ for the clone sets of unlabelled vertices and $t = \Tu(x,y)$ for the tangles in \eqref{eq:BipTangles}.
\end{proof}

\begin{theorem}
\label{thm:genfcnbiplab}
  Let $\Bl(x, y) \in \QQ[[x, y]]$ be the exponential generating function for labelled bicoloured graphs, that is, the formal power series
  \[\Bl(x, y) = \sum_{i, j \geq 0} 2^{ij} \frac{x^i y^j}{i! j!}.\]
  Then, the exponential generating function for labelled tangles is
  \[
    \Tl(x, y) = e^{-x} + e^{-y} - 1 - \Bl(x, y)^{-1}.
  \]
\end{theorem}

\begin{proof}
This follows from \autoref{lem:masterSkelBip} by plugging in the values $c_1= e^x-1$ and $c_2=e^y-1$ for the clone sets of labelled vertices and $t = \Tl(x,y)$ for the tangles in \eqref{eq:BipTangles}.
\end{proof}

With these expressions for the generating functions $\Tu(x,y)$ and $\Tl(x,y)$ in hand, the following corollaries of \autoref{thm:skeleta-genfunc} yield the equations~\eqref{ordgs} and~\eqref{expgs} from the introduction.

\begin{corollary}
\label{cor:genfcn3p1pos}
  Let $S(c, t)$ be the generating function of \autoref{thm:skeleta-genfunc} for skeleta. Then, the ordinary generating function for unlabelled $(3+1)$-free posets is
  \[
    \sum_{n\geq 0} \pu(n) x^n
      = S\big(x/(1-x), \Tu(x, x)\big).
  \]
\end{corollary}

\begin{corollary}
  Let $S(c, t)$ be the generating function of \autoref{thm:skeleta-genfunc} for skeleta. Then, the exponential generating function for labelled $(3+1)$-free posets is
  \[
    \sum_{n\geq 0} \pl(n) \frac{x^n}{n!}
      = S\big(e^x - 1, \Tl(x, x)\big).
  \]
\end{corollary}

\begin{remark}
Fran\c{c}ois Bergeron has pointed out to us that the results of this
  section can be generalized to obtain the cycle index series (see~\cite{BLL}) for the species of $(3+1)$-free posets.
\end{remark}

\section{Acknowledgements}\label{sec:thanks}

This work grew out of a working session of the algebraic combinatorics group at LaCIM with active participation from Chris Berg, Franco Saliola, Luis Serrano, and the authors.
It was facilitated by computer exploration using various types of mathematical software, including Sage~\cite{sage} and its algebraic combinatorics features developed by the Sage-Combinat community~\cite{sage-combinat}.
In addition, the authors would like to thank Joel Brewster Lewis for
several conversations and
suggesting looking at asymptotics and
Mark Skandera and Yan X Zhang for helpful discussions.

\end{document}